\documentclass[11pt]{amsart}
\usepackage{amsmath,amsthm,amscd,amsfonts,amssymb,graphicx,paralist,color}
\usepackage{enumitem}
\usepackage[initials]{amsrefs}
\usepackage{enumerate}
\usepackage[initials]{amsrefs}
\usepackage{xcolor}
\usepackage{hyperref}
\usepackage{rotating}
\hypersetup{colorlinks=false,linkbordercolor=red,linkcolor=green,pdfborderstyle={/S/U/W 1}}
\usepackage[skip=0.7\baselineskip]{caption}

\usepackage[utf8]{inputenc}
\usepackage[T1]{fontenc}

\usepackage{tikz}
\usetikzlibrary{trees}

\theoremstyle{plain} 
\newtheorem{theorem}{Theorem}[section]
\newtheorem{lemma}[theorem]{Lemma} 
 
\newtheorem{proposition}[theorem]{Proposition}

\theoremstyle{definition} 
\newtheorem{example}[theorem]{Example}
\newtheorem{remark}[theorem]{Remark}
\newcommand{\R}{\mathbb{R}}
\newcommand{\C}{\mathbb{C}}
\newcommand{\Q}{\mathbb{Q}}

\newcommand{\F}{\mathbb{F}}

\newcommand{\md}{\mathrm{md}}
\newcommand{\dens}{\mathrm{dens}}

\makeatletter
\@namedef{subjclassname@1991}{2020 Mathematics Subject Classification}
\makeatother

\title[Two Notes on Valued  Fields]{Two Notes on Valued  Fields}

\author[Maghsoudi]{S. Maghsoudi}
\address[S. Maghsoudi]{\mbox{}\newline \indent Department of Mathematics, \newline \indent  University of Zanjan, \newline \indent  Zanjan 45371-38791, Iran.}
\email{s\_maghsodi@znu.ac.ir}

\author[Rodríguez-Vidanes]{Daniel L. Rodríguez-Vidanes}
\address[Daniel L. Rodríguez-Vidanes]{\mbox{}\newline \indent Grupo de Investigación: Análisis Matemático y Aplicaciones, \newline \indent Departamento de Matemática Aplicada a la Ingeniería Industrial, \newline \indent Escuela Técnica Superior de Ingeniería y Diseño Industrial, \newline \indent Universidad Politécnica de Madrid, \newline \indent Ronda de Valencia 3, Madrid, 28012, Spain. \newline
\indent \href{http://orcid.org/0000-0002-1016-096X}{ORCID: \texttt{0000-0002-1016-096X} }}
\email{\texttt{dl.rodriguez.vidanes@upm.es}}

\begin{document}
	
	


	\begin{abstract}
		This paper studies two questions on valued fields: the metric dimension induced by an absolute value, and the uniform openness of multiplication. 
		For nontrivial non-archimedean absolute values, we prove that the metric dimension equals the density character. 
		In the archimedean case, the metric dimension is 2 for subfields of $\mathbb R$, while for non-real subfields of $\mathbb C$ it is either 2 or 3, according to whether the field contains a non-real element together with its complex conjugate.
		We also show that multiplication is uniformly open on every valued field. 
		Finally, we prove that this property is genuinely metric, not purely topological, even on $\mathbb{R}$ with a suitable compatible choice of metric.
	\end{abstract}
	
	\maketitle

\section{Introduction and Preliminaries}

Throughout the paper, $\F$ is a field endowed with an absolute value $|\cdot|$, i.e., $\F$ is a valued field. Recall that by an absolute value on a field $\F$ we mean a function $|\cdot|: \F \to [0,\infty)$ satisfying 
\begin{itemize} 
    \item[(i)] $|x| = 0$ if and only if $x = 0$,
    \item[(ii)] $|xy| = |x||y|$ for all $x,y \in \F$,
    \item[(iii)] $|x+y| \le |x| + |y|$ for all $x,y \in \F$.
\end{itemize}
The absolute value is \emph{non-archimedean} if it satisfies the stronger inequality:
\begin{itemize}
    \item[(iii')] $|x+y| \le \max(|x|, |y|)$ for all $x,y \in \F$.
\end{itemize}
Recall that the topology induced by an absolute value on a field is discrete if and only if the absolute value is trivial, i.e., $|x|=1$ if $x\neq 0$ and $|0|=0$.
Also, notice that the trivial absolute value is non-archimedean.

Following S. Bau and A. F. Beardon in \cite{1}, we define the metric dimension of a metric space $(X,d)$ as follows. A subset $A \subseteq X$ is called a \emph{resolving set} for $X$ if the map
	\[
	x \longmapsto (d(x,a))_{a \in A}
	\]
from $X$ to $\R^{A}$ is injective. That is, if $d(x,a) = d(y,a)$ for all $a \in A$, then $x = y$. The \emph{metric dimension} of $X$, denoted $\md(X)$, is the smallest cardinality of a resolving set for $X$.

The concept of metric dimension, which quantifies the minimum number of landmarks needed to uniquely determine positions in a metric space, has applications ranging from navigation systems to graph theory. The concept of the metric dimension of metric spaces was introduced in 1953 by L.M. Blumenthal \cite{3}. Twenty years later, F. Harary, R. A. Melter, and P. J. Slater applied it to the metric spaces defined by graphs \cite{7,12}. In \cite{1}   the metric dimension of a sphere in $n$-dimensional Euclidean space has been calculated; for generalizations of these results in the same line see \cite{hm} and references therein. It is known, for instance, that $\md(U)=n+1$ for any non-empty open subset $U\subseteq \R^n$ and $\md(V)={\rm dim}(V)+1$ for any affine subspace of $\R^n$, where ${\rm dim}(V)$ denotes the algebraic dimension of the vector space $V$, see \cite{1}.
More recently, the metric dimension has continued to be studied beyond the setting of graphs and Euclidean spaces.
For instance, Cai and Liao in \cite{CaiLiao2026} computed the metric dimensions of the complex hyperbolic space $\mathbb H^n_{\mathbb C}$ and of the Heisenberg group $\mathcal H$, obtaining the values $2n+1$ and $4$, respectively.
In a different direction, Lladser and Paradise in \cite{LladserParadise2026} studied the resolvability of finite Jaccard spaces $(2^X,\operatorname{Jac})$ and proved that their metric dimension is of order $\Theta(|X|/\ln |X|)$.

Motivated by the classical Banach open mapping principle, in \cite{str} a function $f : X \to Y$ is called \emph{uniformly open} if for every $\varepsilon > 0$, there exists $\delta > 0$ such that for every $x \in X$,
\[
B_{Y}( f(x), \delta ) \subseteq f( B_{X}( x, \varepsilon ) ),
\]
where $B_X(x, \varepsilon)$ is the open ball in $X$ of radius $\varepsilon$ centered at $x$, and similarly for $B_Y$. The problem of uniform openness of multiplication in certain algebras was posed and studied in \cite{str, dra}.

Our aim in this paper is twofold. 
First, we determine the metric dimension of valued fields endowed with the metric induced by an absolute value. 
In the non-archimedean case, we show that for a nontrivial absolute value it coincides with the density character of the field, whereas in the archimedean case it is equal to $2$ or $3$ according to the position of the field inside $\R$ or $\C$ and, in the non-real case, according to whether the field contains a non-real element together with its complex conjugate.
Secondly, we prove that multiplication is uniformly open on every valued field. 
We also show that this phenomenon is genuinely metric, by exhibiting a compatible metric on the usual field topology of $\R$ for which multiplication is open but not uniformly open.

 We shall use standard set-theoretical notation. 
 As usual,   $\mathbb{Q}$, $\mathbb{R}$  and $\mathbb{C}$ denote the sets of all   rational, real  and complex numbers, respectively. 	
 We write ${\rm card}(A)$ or the cardinality of a set $A$, and we set $\aleph_0={\rm card}(\mathbb N)$.
 Also, the complex conjugate and the real part of $z\in\C$ are denoted by $\overline z$ and $\Re(z)$, respectively.

\section{The Metric Dimension}

In this section, we determine the metric dimension of a valued field. To begin,  we need the following two elementary observations. First, let $(X,d)$ be a metric space, $c > 0$, and define $d_c(x,y) := d(x,y)^c$. Then, $A \subseteq X$ is resolving for $(X,d)$ if and only if it is resolving for $(X,d_c)$. Secondly, by Ostrowski's theorem \cite[Theorem 15.2]{she}, any archimedean absolute value on a field $\F$ is equivalent to one of the form $ |x| = |\sigma(x)|_\infty^c $, where $\sigma: \F \to \C$ is an isomorphism, $|\cdot|_\infty$ is the usual absolute value on $\C$, and $c > 0$.

\begin{proposition}\label{thm:main}
	Let $\F$ be a field endowed with an archimedean absolute value $|\cdot|$. Then:
	\begin{itemize}
		\item If $\sigma(\F)\subseteq \R$ (as a metric subspace), then $\md(\F)=2$.
		\item If $\sigma(\F)\nsubseteq \R$, then $2\leq \md(\F)\le 3$. 
		Furthermore, $\md(\F)=3$ if and only if there exists $z\in \sigma(\F)\setminus \R$ such that $\overline z\in \sigma(\F)$.
	\end{itemize}
\end{proposition}

\begin{proof}
	Let $d(x,y)=|x-y|=|\sigma(x)-\sigma(y)|_\infty^c$. 
	As mentioned above, the resolving sets for $(\F,d)$ are the same as for $(\F,|\sigma(\cdot)-\sigma(\cdot)|_\infty)$. 
	Thus, we may assume $c=1$ and identify $\F$ with its image $\sigma(\F)\subseteq \C$, endowed with the usual absolute value. 
	We then consider two cases.
	
	\medskip
	
	\textit{Case 1:} $\F\subseteq \R$. 
	On the one hand, we will show that $\{0,1\}$ is a resolving set for $\F$. 
	Let $x,y\in \F$ satisfy $|x|=|y|$ and $|x-1|=|y-1|$.
	Since $x,y\in \R$, from $|x|=|y|$ we get either $x=y$ or $x=-y$.
	In the latter case,
		\[
		|x-1|=|-y-1|=|y+1|,
		\]
	so $|y+1|=|y-1|$, which implies $y=0$, and hence again $x=y$. 
	Therefore $\{0,1\}$ resolves $\F$, and so $\md(\F)\le 2$.
	
	On the other hand, no one-point set resolves $\F$. 
	Indeed, for every $a\in \F$, the points $a+1$ and $a-1$ belong to $\F$, are distinct, and satisfy
		\[
		|(a+1)-a|=|(a-1)-a|=1.
		\]
	Hence, $\md(\F)>1$, and therefore $\md(\F)=2$.
	
	\medskip
	
	\textit{Case 2:} $\F\nsubseteq \R$. 
	Choose $\alpha\in \F\setminus \R$. 
	We will first show that $\{0,1,\alpha\}$ is a resolving set for $\F$, and hence that $\md(\F)\le 3$.
	
	Let $x,y\in \F$ satisfy $|x|=|y|$, $|x-1|=|y-1|$ and $|x-\alpha|=|y-\alpha|$.
	From the first two equalities we obtain
		\[
		|x|^2-|x-1|^2=|y|^2-|y-1|^2.
		\]
	Since
		\[
		|x-1|^2=|x|^2-2\Re(x)+1
		\]
	and similarly for $y$, it follows from $|x-1|=|y-1|$ that
		\[
		\Re(x)=\Re(y).
		\]
	Together with $|x|=|y|$, this implies that either $y=x$ or $y=\overline x$.
	
	If $y=\overline x\neq x$, then
		\[
		|x-\alpha|=|\overline x-\alpha|.
		\]
	But the set of points in $\C$ equidistant from $x$ and $\overline x$ is exactly the real axis, so this would force $\alpha\in \R$, a contradiction. Therefore $y=x$, and $\{0,1,\alpha\}$ resolves $\F$. Thus $\md(\F)\le 3$.
	
	\medskip
	
	We will now prove the equivalence.
	Assume first that there is no $z\in \F\setminus \R$ such that $\overline z\in \F$. We claim that $\{0,1\}$ resolves $\F$. Let $x,y\in \F$ satisfy $|x|=|y|$ and $|x-1|=|y-1|$.
	As above, these two equalities imply that either $y=x$ or $y=\overline x$. 
	If $y=\overline x$ and $x\notin \R$, then $\overline x\in \F$, contradicting the assumption. 
	If $x\in \R$, then $\overline x=x$ anyway. Hence $y=x$ in all cases, so $\{0,1\}$ resolves $\F$. 
	Therefore $\md(\F)\le 2$. 
	Since no one-point set resolves $\F$, we conclude that $\md(\F)=2$.
	
	Conversely, assume that there exists $u\in \F\setminus \R$ such that $\overline u\in \F$. 
	We will show that no two-point subset of $\F$ resolves $\F$. 
	Let $p,q\in \F$ with $p\neq q$, and define
		\[
		x:=p+u(q-p), \qquad y:=p+\overline u\,(q-p).
		\]
	Then, $x,y\in \F$, and since $u\notin \R$ we have $u\neq \overline u$, so $x\neq y$. 
	Moreover,
		\[
		|x-p|=|u|\,|q-p|=|\overline u|\,|q-p|=|y-p|.
		\]
	Also,
		\[
		x-q=(u-1)(q-p), \qquad y-q=(\overline u-1)(q-p),
		\]
	so
		\[
		|x-q|=|u-1|\,|q-p|=|\overline u-1|\,|q-p|=|y-q|.
		\]
	Thus, $\{p,q\}$ does not resolve $\F$. 
	Since $p,q$ were arbitrary, no two-point set resolves $\F$, so $\md(\F)> 2$. 
	Together with $\md(\F)\le 3$, this yields $\md(\F)=3$.
	This completes the proof.
\end{proof}

Let us provide two archimedean valued fields not contained in the real line, one with metric dimension $2$ and the other with metric dimension $3$.

\begin{example}
	\noindent 
	\begin{itemize}
		\item The case $\md(\F)=3$. 
		Let $\F=\Q(i)$. Then $\F\nsubseteq \R$, and complex conjugation preserves $\F$, since for every $a,b\in \Q$ we have $\overline{a+bi}=a-bi\in \Q(i)$.
		In particular, there exists a non-real element $z\in \F$ such that $\overline z\in \F$.
		Hence, by Proposition~\ref{thm:main}, we have $\md(\Q(i))=3$.
		
		\item The case $\md(\F)=2$. 
		Let $\alpha$ be a non-real root of the polynomial
		$$
		p(x)=x^3-x+1,
		$$
		and set $\F=\Q(\alpha)$. 
		Then, $\F\nsubseteq \R$. 
		Moreover, $p$ is irreducible over $\Q$ by the rational root test, and its discriminant is $\Delta(p)=-23$, which is not a square in $\Q$. 
		Thus, the cubic extension $\Q(\alpha)/\Q$ is not Galois.
		
		We claim that there is no non-real element $z\in \F$ such that $\overline z\in \F$.
		Indeed, assume that such a $z$ existed.
		Since $[\F:\Q]=3$ is prime and $z\notin \Q$, we would have $\F=\Q(z)$.
		As $\overline z\in \F$ and $\overline z\notin \Q$, we would also have $\F=\Q(\overline z)$.
		Therefore, the assignment $z\mapsto \overline z$ would induce a nontrivial $\Q$-automorphism of $\F$.
		Since $[\F:\Q]=3$ is prime, the existence of such a nontrivial automorphism would force $\F/\Q$ to be Galois, a contradiction.
		
		Hence, no non-real element of $\F$ has its complex conjugate in $\F$.
		By Proposition~\ref{thm:main}, we conclude that $\md(\Q(\alpha))=2$.
	\end{itemize}
\end{example}

\begin{remark}
    The metric dimension depends on the metric, not just the topology. 
    For instance, $\R$ with the bounded metric $d(x,y)=\min\{|x-y|,1\}$ is homeomorphic to the Euclidean real line, but its metric dimension is infinite because no finite set can resolve points that are sufficiently far apart.
\end{remark}

The \emph{density character} of a topological space $X$, denoted $\dens(X)$, is the smallest cardinality of a dense subset of $X$.

\medskip

It is clear that given $a,x,y$ in an ultrametric space $(X,d)$, if $d(x,y) < d(a,x)$, then $d(a,x) = d(a,y)$.
We will also use the fact that a field endowed with a nontrivial absolute value has no isolated points.
Indeed, if $0<|\tau|<1$, then $\tau^n\to 0$, and hence $x+\tau^n\to x$ with $x+\tau^n\neq x$.

\begin{proposition}\label{thm:finite-fails}
	Let $\F$ be an infinite field endowed with a nontrivial non-archimedean absolute value. 
	Then, no finite subset of $\F$ is resolving.
\end{proposition}


\begin{proof}
	Let $A\subseteq \F$ be finite.
	If $A=\varnothing$, then $A$ is clearly not resolving, since $\F$ contains at least two distinct points.
	Thus, we may assume that $A\neq\varnothing$.

	Let
		\[
		R:=\max\{|a-b|:a,b\in A\}
		\]
	be the diameter of $A$. 
	Set
		\[
		M:=\max\left(R,\max_{a\in A}|a|\right).
		\]
	Since the value group is unbounded above (see \cite[p.~4]{van}), we may choose $x\in \F$ such that $|x|>M$.
	Then, for every $a\in A$, we have $|x|>|a|$, and therefore the ultrametric inequality yields
		\[
		|x-a|=\max\{|x|,|a|\}=|x|>R,
		\]
	for every $a\in A$.
	
	Now set $\varepsilon:=|x|>0$.
	Since $\F$ has no isolated points, we can choose $y\in \F$ with $y\neq x$ and $|x-y|<\varepsilon$.
	For each $a\in A$, we then have $|x-y|<|x-a|$.
	Hence, again by the ultrametric inequality, $|y-a|=|x-a|$ for every $a\in A$, while $x\neq y$. 
	Therefore, $A$ does not resolve $\F$.
	
	Since $A$ was arbitrary, no finite subset of $\F$ is resolving.
\end{proof}

\begin{proposition}\label{thm:dense-works}
	Let $\F$ be a field endowed with a nontrivial non-archimedean absolute value. 
	Then, every dense subset $A\subseteq \F$ is resolving.
\end{proposition}


\begin{proof}
	Let $x,y\in \F$ with $x\neq y$, and set $t:=y-x\neq 0$.
	We shall find $a\in A$ such that $|x-a|\neq |y-a|$.
	
	Since $A$ is dense in $\F$, and $B(x,|t|)=\{z\in \F:|z-x|<|t|\}$ is an open neighbourhood of $x$, there exists $a\in A$ such that $|x-a|<|t|$.
	Now $y-a=(x-a)+t$.
	Since $|x-a|<|t|$, the ultrametric inequality implies
		\[
		|y-a|=|(x-a)+t|=|t|.
		\]
	Therefore,
		\[
		|x-a|<|t|=|y-a|,
		\]
	so in particular $|x-a|\neq |y-a|$.
	
	Hence, $A$ separates the points $x$ and $y$. Since $x\neq y$ were arbitrary, $A$ is a resolving set for $\F$.
\end{proof}

\begin{theorem}\label{thm:main-nonarch}
	Let $\F$ be an infinite field endowed with a nontrivial non-archimedean absolute value. Then,
	\[
	\md(\F)=\dens(\F).
	\]
	Moreover, the same equality holds for the trivial absolute value, in which case
	\[
	\md(\F)=\dens(\F)=\operatorname{card}(\F).
	\]
\end{theorem}

%
%

\begin{proof}
	Assume first that $|\cdot|$ is nontrivial and non-archimedean.
	
	By Proposition~\ref{thm:finite-fails}, no finite subset of $\F$ is resolving. In particular, $\md(\F)$ is infinite. On the other hand, by Proposition~\ref{thm:dense-works}, every dense subset of $\F$ is resolving. Therefore,
		\[
		\md(\F)\le \dens(\F).
		\]
	
	For the reverse inequality, let $A\subseteq \F$ be a resolving set. We claim that $A$ must be dense. 
	Assume, on the contrary, that $A$ is not dense. 
	Then, there exist $x\in \F$ and $r>0$ such that $B(x,r)\cap A=\varnothing$.
	Choose two distinct points $y,z\in B(x,r)$. Since $a\notin B(x,r)$ for every $a\in A$, we have
		\[
		|x-a|\ge r>|x-y|
		\qquad\text{and}\qquad
		|x-a|\ge r>|x-z|.
		\]
	By the ultrametric inequality, it follows that
		\[
		|y-a|=|(y-x)+(x-a)|=|x-a|
		\]
	and similarly
		\[
		|z-a|=|(z-x)+(x-a)|=|x-a|.
		\]
	Hence,
		\[
		|y-a|=|z-a|,
		\]
	for every $a\in A$,
	which contradicts the fact that $A$ resolves $\F$. 
	Thus, every resolving set is dense, and so
		\[
		\dens(\F)\le \md(\F).
		\]
	
	Combining both inequalities, we conclude that
		\[
		\md(\F)=\dens(\F).
		\]
	
	Finally, assume that $|\cdot|$ is the trivial absolute value. 
	Then the induced topology on $\F$ is discrete, so
		\[
		\dens(\F)=\operatorname{card}(\F).
		\]
	We claim that a subset $A\subseteq \F$ is resolving if and only if $\operatorname{card}(\F\setminus A)\le 1$.
	
	Indeed, if there exist distinct $y,z\in \F\setminus A$, then for every $a\in A$ we have
		\[
		|y-a|=|z-a|=1,
		\]
	so $A$ does not resolve $\F$. Conversely, if $\operatorname{card}(\F\setminus A)\le 1$, then for any distinct $x,y\in \F$, at least one of them belongs to $A$, say $x\in A$. 
	Then,
		\[
		|x-x|=0\neq |y-x|,
		\]
	so $A$ resolves $\F$.
	
	Hence, every resolving set has complement of cardinality at most one, and since $\F$ is infinite, any such set has cardinality $\operatorname{card}(\F)$. 
	Therefore,
		\[
		\md(\F)=\operatorname{card}(\F)=\dens(\F).
		\]
	This concludes the proof.
\end{proof}

\begin{example}
    By Theorem~\ref{thm:main-nonarch}, $\md(\mathbb{Q}_p) = \aleph_0$ and $\md(\mathbb{F}_q((t))) = \aleph_0$ since both are separable. 
    As  $\mathbb{C}_{p}$
    is a separable metric space \cite[p. 140]{rob}, its density character is $\aleph_{0}$. 
    By Theorem~\ref{thm:main-nonarch} we therefore obtain $\md(\mathbb{C}_{p})=\aleph_{0}$.  
    Here $\mathbb{Q}_p$, $\mathbb{C}_p$ and $\mathbb{F}_q((t))$ denote the fields of $p$-adic numbers, $p$-adic complex numbers, and the field of formal Laurent series over the finite field $\mathbb{F}_q$, respectively.
\end{example}

\section{Uniform Openness of Multiplication}

We now turn to the problem of uniform openness of multiplication in a valued field.
Throughout this section, products of metric spaces will be endowed with the maximum metric; that is, if $(X,d_X)$ and $(Y,d_Y)$ are metric spaces, we equip $X\times Y$ with the metric
	$$
	d((x,y),(s,t))=\max\{d_X(x,s),d_Y(y,t)\}.
	$$

\begin{lemma}\label{met}
	Let $(\F_1, |\cdot|_1)$ and $(\F_2, |\cdot|_2)$ be valued fields, and let $\phi : \F_1 \to \F_2$ be an isomorphism of topological fields. 
	If multiplication $m_1:\F_1\times \F_1\to \F_1$
	is uniformly open, then multiplication $m_2:\F_2\times \F_2\to \F_2$ is uniformly open.
\end{lemma}

\begin{proof}
	Since $\phi$ is a topological field isomorphism, both $\phi$ and $\phi^{-1}$ are continuous at $0$. Moreover,
		\[
		m_2=\phi\circ m_1\circ (\phi\times \phi)^{-1}.
		\]
	
	Fix $\varepsilon>0$. By continuity of $\phi$ at $0$, there exists $\varepsilon_1>0$ such that $\phi(B_1(0,\varepsilon_1))\subseteq B_2(0,\varepsilon)$.
	Since $m_1$ is uniformly open, there exists $\delta_1>0$ such that for all $a,b\in \F_1$, we have $B_1(m_1(a,b),\delta_1)\subseteq m_1\left(B_1(a,\varepsilon_1)\times B_1(b,\varepsilon_1)\right)$.
	By continuity of $\phi^{-1}$ at $0$, there exists $\delta_2>0$ such that $\phi^{-1}(B_2(0,\delta_2))\subseteq B_1(0,\delta_1)$.
	
	We claim that $\delta_2$ works for $\varepsilon$. 
	Let $p,q\in \F_2$, and write 
		$$
		a:=\phi^{-1}(p)\qquad \text{and} \qquad b:=\phi^{-1}(q).
		$$
	Take any 
		$$
		z\in B_2(m_2(p,q),\delta_2).
		$$
	Then, $z-\phi(m_1(a,b))\in B_2(0,\delta_2)$, and so $\phi^{-1}(z)-m_1(a,b)\in B_1(0,\delta_1)$,
	that is, $\phi^{-1}(z)\in B_1(m_1(a,b),\delta_1)$.
	Hence, $\phi^{-1}(z)\in m_1\left(B_1(a,\varepsilon_1)\times B_1(b,\varepsilon_1)\right)$.
	Therefore, there exist $u\in B_1(a,\varepsilon_1)$ and $v\in B_1(b,\varepsilon_1)$ such that $\phi^{-1}(z)=m_1(u,v)$.
	Applying $\phi$, we obtain $z=m_2(\phi(u),\phi(v))$.
	Now $u-a\in B_1(0,\varepsilon_1)$, so $\phi(u)-p=\phi(u-a)\in B_2(0,\varepsilon)$, and thus $\phi(u)\in B_2(p,\varepsilon)$.
	Similarly, $\phi(v)\in B_2(q,\varepsilon)$.
	It follows that
		\[
		z\in m_2\left(B_2(p,\varepsilon)\times B_2(q,\varepsilon)\right).
		\]
	
	Thus,
		\[
		B_2(m_2(p,q),\delta_2)\subseteq m_2\left(B_2(p,\varepsilon)\times B_2(q,\varepsilon)\right)
		\]
	for all $p,q\in \F_2$, and hence $m_2$ is uniformly open.
\end{proof}

\begin{theorem}\label{thm:mult-open}
	If $\F$ is a nontrivial valued field, then the multiplication map $m:\F\times\F\to\F$ given by $m(x,y)=xy$ is uniformly open.
\end{theorem}

\begin{proof}
	Fix $\varepsilon>0$. 
	We will construct $\delta>0$ such that for all $a,b\in\F$,
		\[
		B(ab,\delta)\subseteq m\bigl(B(a,\varepsilon)\times B(b,\varepsilon)\bigr).
		\]
	Equivalently, whenever $w\in\F$ satisfies $|w-ab|<\delta$, we will find $u,v\in\F$ such that $|u|,|v|<\varepsilon$ and $(a+u)(b+v)=w$.
	Set $\eta:=w-ab$.
	
	We consider separately the non-archimedean and archimedean cases.
	
	\bigskip
	
	\noindent\textit{Non-archimedean case.}
	Since $\F$ is nondiscrete, there exists $\tau\in\F$ such that $0<|\tau|<\varepsilon$.
	Define $\delta:=|\tau|\,\varepsilon$.
	Let $a,b\in\F$ and assume that $|\eta|<\delta$.
	
	We distinguish three cases.
	
	\medskip
	
	\noindent\textbf{Case 1:} $|a|\ge |\tau|$ (analogously for $|b|\geq |\tau|$).
	
	Set
		\[
		u:=0 \qquad \text{and} \qquad v:=\frac{\eta}{a}.
		\]
	Then,
		\[
		(a+u)(b+v)=a\left(b+\frac{\eta}{a}\right)=ab+\eta=w,
		\]
	and
		\[
		|v|=\frac{|\eta|}{|a|}<\frac{\delta}{|\tau|}=\varepsilon.
		\]
	Thus, $a+u\in B(a,\varepsilon)$ and $b+v\in B(b,\varepsilon)$.
	
	\medskip
	
	\noindent\textbf{Case 2:} $|a|<|\tau|$ and $|b|<|\tau|$.
	
	Set
		\[
		u:=\tau \qquad \text{and} \qquad v:=\frac{\eta-b\tau}{a+\tau}.
		\]
	Since $|a|<|\tau|$, the ultrametric inequality gives
		\[
		|a+\tau|=|\tau|\neq 0,
		\]
	so $v$ is well defined. Moreover,
		\[
		(a+u)(b+v)=(a+\tau)\left(b+\frac{\eta-b\tau}{a+\tau}\right)=ab+b\tau+\eta-b\tau=ab+\eta=w.
		\]
	Also, $|\eta-b\tau|\le \max\{|\eta|,\ |b\tau|\}$.
	Now, $|\eta|<\delta=|\tau|\varepsilon$ and 
		$$
		|b\tau|=|b|\,|\tau|<|\tau|^2<|\tau|\,\varepsilon
		$$
	since $|\tau|<\varepsilon$. 
	Hence,
		\[
		|\eta-b\tau|<|\tau|\,\varepsilon,
		\]
	and therefore,
	\[
	|v|=\frac{|\eta-b\tau|}{|a+\tau|}=\frac{|\eta-b\tau|}{|\tau|}<\varepsilon.
	\]
	Since $|u|=|\tau|<\varepsilon$, we again obtain $a+u\in B(a,\varepsilon)$, $b+v\in B(b,\varepsilon)$ and $(a+u)(b+v)=w$, concluding this case.
	
	Thus multiplication is uniformly open in the non-archimedean case.
	
	\medskip
	
	\noindent\textit{Archimedean case.}
	By Ostrowski's theorem, there exist $c>0$ and a field embedding $\sigma:\F\hookrightarrow \C$ such that $|x|=|\sigma(x)|_\infty^{\,c}$ for all $x\in \F$, where $|\cdot|_\infty$ denotes the usual absolute value on $\C$.
	
	Let $d(z,t):=|z-t|_\infty$ and $d_c(z,t):=d(z,t)^c$. Since
		\[
		B_{d_c}(z,r)=B_d\bigl(z,r^{1/c}\bigr),
		\]
	multiplication is uniformly open with respect to $d_c$ if and only if it is uniformly open with respect to $d$. Therefore, by Lemma~\ref{met}, it is enough to prove the theorem when $\F$ is regarded as a subfield of $\C$ endowed with the usual absolute value. We therefore assume from now on that $\F\subseteq\C$ and that $|\cdot|$ is the usual absolute value.
	
	If $\F\subseteq\R$, set $\beta:=-1$.
	If $\F\nsubseteq\R$, choose any $\beta\in\F\setminus\R$.
	In either case, $\beta\in\F$ and $\beta\neq 1$. 
	Define
		\[
		M:=\max\{1,|\beta|\} \qquad \text{and} \qquad \kappa:=\frac{|1-\beta|}{2}>0.
		\]
	Choose a rational number $r>0$ such that
		\[
		Mr<\frac{\varepsilon}{2},
		\]
	and set
	\[
	\delta:=\frac{\kappa r\varepsilon}{2}.
	\]
	
	Let $a,b\in\F$ and assume that $|\eta|<\delta$.
	
	We again split into cases.
	
	\medskip
	
	\noindent\textbf{Case 1:} $|a|\ge \kappa r$ (analogously for $|b|\geq \kappa r$).
	
	Set
		\[
		u:=0,\qquad v:=\frac{\eta}{a}.
		\]
	Then
		\[
		(a+u)(b+v)=a\left(b+\frac{\eta}{a}\right)=ab+\eta=w,
		\]
	and
		\[
		|v|=\frac{|\eta|}{|a|} < \frac{\delta}{\kappa r} = \frac{\varepsilon}{2} < \varepsilon.
		\]

	\medskip
	
	\noindent\textbf{Case 2:} $|a|<\kappa r$ and $|b|<\kappa r$.
	
	Consider 
		\[
		u_1:=r \qquad \text{and} \qquad u_2:=r\beta.
		\]
	We claim that at least one of them satisfies
		\[
		|a+u_j|\ge \kappa r.
		\]
	Indeed, if both
		\[
		|a+r|<\kappa r \qquad\text{and}\qquad |a+r\beta|<\kappa r
		\]
	held, then the triangle inequality would give
		\[
		r|1-\beta|=|r-r\beta|=|(a+r)-(a+r\beta)|\le |a+r|+|a+r\beta|<2\kappa r=r|1-\beta|,
		\]
	a contradiction.
	
	Choose $u\in\{r,r\beta\}$ such that $|a+u|\ge \kappa r$. 
	Since $|u|\le Mr<\varepsilon/2<\varepsilon$, we have $a+u\in B(a,\varepsilon)$. 
	Define
		\[
		v:=\frac{\eta-bu}{a+u}.
		\]
	Then,
		\[
		(a+u)(b+v)=(a+u)\left(b+\frac{\eta-bu}{a+u}\right)=ab+bu+\eta-bu=ab+\eta=w.
		\]
	Moreover,
		\[
		|v|\le \frac{|\eta|+|b|\,|u|}{|a+u|}< \frac{\delta+\kappa r\cdot Mr}{\kappa r}= \frac{\varepsilon}{2}+Mr<\varepsilon.
		\]
	Hence, $b+v\in B(b,\varepsilon)$.
	This proves uniform openness in the archimedean case as well.
\end{proof}

\begin{remark}
	\noindent
\begin{itemize}
\item[(1)] It is clear that if $\F$ is endowed with the trivial absolute value, then multiplication is uniformly open.
Indeed, for every $\varepsilon>0$, one may take $\delta=1$.

\item[(2)] Theorem~\ref{thm:mult-open} also applies, with the appropriate absolute-value metric, to every nondiscrete locally compact topological field $\F$.
Indeed, by the Pontryagin--Kowalsky--van Dantzig theorem (see \cite[Theorem~6.2.3]{sil}), every nondiscrete locally compact topological field is isomorphic, as a topological field, to $\R$, $\C$, a finite extension of $\Q_p$, or a field of formal Laurent series $\F_q((t))$.
In each of these cases, the topology is induced by a nontrivial absolute value $|\cdot|$ which makes the field complete.
Thus, after transporting this absolute-value metric through the topological-field isomorphism, Theorem~\ref{thm:mult-open} yields the uniform openness of multiplication.

\item[(3)] Since uniform openness depends on the chosen compatible metric and not only on the underlying topology, Theorem~\ref{thm:mult-open} cannot be reformulated purely in topological terms. 
In fact, even the usual field topology on $\R$ admits a compatible metric for which multiplication is open but not uniformly open as the following proposition shows.
\end{itemize}
\end{remark}

\begin{proposition}
	There exists a metric $d$ on $\R$, compatible with the usual topology, such that multiplication $m:\R\times\R\to\R$ given by $m(x,y)=xy$ is open but not uniformly open.
\end{proposition}

\begin{proof}
	Choose a sequence $(t_n)_{n\in\mathbb N}$ of positive real numbers such that $t_1>3$ and $t_{n+1}>t_n^2+3$, for every $n\in\mathbb N$.
	For every $n\in\mathbb N$, set
		$$
		M_n:=t_n^4,	\qquad \mu_n:=t_n^{-4},
		$$
	and
		$$
		I_n:=[t_n-1,t_n+1] \qquad \text{and} \qquad J_n:=[t_n^2-1,t_n^2+1].
		$$
	Since $t_n>3$, we have $t_n+1<t_n^2-1$, and hence $I_n$ lies strictly to the left of $J_n$.
	Moreover, by the choice of the sequence $(t_n)$, $t_n^2+1<t_{n+1}-1$, so $J_n$ lies strictly to the left of $I_{n+1}$.
	Thus, the intervals $I_n$ and $J_n$ are pairwise disjoint.
	
	We may therefore choose an increasing piecewise linear homeomorphism $h:\mathbb R\to\mathbb R$ such that, for every $n\in\mathbb N$,
		\begin{itemize}
			\item the slope of $h$ on $I_n$ is $M_n$,
			\item the slope of $h$ on $J_n$ is $\mu_n$,
			\item the slope of $h$ outside $\bigcup_n (I_n\cup J_n)$ is $1$.
		\end{itemize}
	Define
		$$
		d(x,y):=|h(x)-h(y)|
		$$
	for every $x,y\in\mathbb R$.
	Since $h$ is a homeomorphism, $d$ is compatible with the usual topology on $\mathbb R$.
	As $d$ induces the usual topology and multiplication is open for the usual topology on $\mathbb R$, multiplication is also open with respect to $d$.
	
	We claim, however, that multiplication is not uniformly open with respect to $d$.
	Fix $\varepsilon=1/2$.
	For each $n\in\mathbb N$, the function $h$ has slope $M_n$ on $I_n$, and therefore
			\begin{equation}\label{equ:1}
				|h(t_n\pm 1)-h(t_n)|=M_n>\varepsilon.
			\end{equation}
	We will now show that
		$$
		B_d(t_n,\varepsilon)=\left(t_n-\frac{\varepsilon}{M_n},\,t_n+\frac{\varepsilon}{M_n}\right).
		$$
	Indeed, if $x\in B_d(t_n,\varepsilon)$, then $|h(x)-h(t_n)|<\varepsilon$.
	Since $h$ is increasing and the points $t_n\pm 1$ satisfy \eqref{equ:1}, it follows that necessarily $x\in(t_n-1,t_n+1)\subset I_n$.
	On this interval, $h$ is affine with slope $M_n$.
	Therefore, $|h(x)-h(t_n)|=M_n|x-t_n|$.
	Thus, $|h(x)-h(t_n)|<\varepsilon$ if and only if $|x-t_n|<\frac{\varepsilon}{M_n}$, which proves the claimed equality.
	
	Let $x=t_n+\alpha$ and $y=t_n+\beta$ with $|\alpha|<\frac{\varepsilon}{M_n}$ and $|\beta|<\frac{\varepsilon}{M_n}$.
	Then, $xy-t_n^2=t_n(\alpha+\beta)+\alpha\beta$,	so
		$$
		|xy-t_n^2| \leq 2t_n\frac{\varepsilon}{M_n}+\frac{\varepsilon^2}{M_n^2} =: \rho_n.
		$$
	Since $M_n=t_n^4$, we have
		$$
		\rho_n=\frac{2\varepsilon}{t_n^3}+\frac{\varepsilon^2}{t_n^8} \xrightarrow[n\to\infty]{}0.
		$$
	In particular, there exists $n_0\in\mathbb N$ such that $\rho_n<1$ for every $n>n_0$.
	For such $n$, every product $xy$ with $x,y\in B_d(t_n,\varepsilon)$ belongs to $J_n$.
	On $J_n$ the slope of $h$ is $\mu_n$, hence
	$$
	d(xy,t_n^2)=|h(xy)-h(t_n^2)|
	=\mu_n|xy-t_n^2|
	\leq \mu_n\rho_n.
	$$
	Therefore,
	$$
	m\left(B_d(t_n,\varepsilon)\times B_d(t_n,\varepsilon)\right)
	\subseteq B_d(t_n^2,\delta_n),
	$$
	where
	$$
	\delta_n:=\mu_n\rho_n
	=\frac{2\varepsilon}{t_n^7}+\frac{\varepsilon^2}{t_n^{12}}
	\xrightarrow[n\to\infty]{}0.
	$$
	
	Assume now that multiplication were uniformly open.
	Then, for the above $\varepsilon>0$, there would exist $\delta>0$ such that, for all $p,q\in\mathbb R$,
	$$
	B_d(pq,\delta)\subseteq m\bigl(B_d(p,\varepsilon)\times B_d(q,\varepsilon)\bigr).
	$$
	Choose $n>n_0$ large enough so that $\delta_n<\delta$.
	Then
	$$
	m\bigl(B_d(t_n,\varepsilon)\times B_d(t_n,\varepsilon)\bigr)
	\subseteq B_d(t_n^2,\delta_n)
	\subsetneq B_d(t_n^2,\delta),
	$$
	which contradicts the previous inclusion applied with $p=q=t_n$.
	Thus, multiplication is not uniformly open with respect to $d$.
\end{proof}



\end{document}